\documentclass[12pt]{amsart}
\usepackage{amssymb, amscd, amsmath, float, graphicx, longtable,color, enumerate, fullpage}

\usepackage[all]{xy}
\SelectTips{eu}{} 
\xyoption{curve}

\newcommand{\hide}[1]{}

\usepackage[colorlinks=true,
]{hyperref}

\usepackage{lipsum}

\makeatletter
\newcommand\ackname{Acknowledgements}
\if@titlepage
  \newenvironment{acknowledgements}{%
      \titlepage
      \null\vfil
      \@beginparpenalty\@lowpenalty
      \begin{center}%
        \bfseries \ackname
        \@endparpenalty\@M
      \end{center}}%
     {\par\vfil\null\endtitlepage}
\else
  
\fi
\makeatother

\renewcommand{\AA}{\mathbb A}
\newcommand{\NN}{{\mathbb N}}
\newcommand{\ZZ}{{\mathbb Z}}
\newcommand{\KK}{{\mathbb K}}
\newcommand{\SSS}{{\mathbb S}}

\newcommand{\bdot}{\bullet}
\DeclareMathOperator{\Der}{Der}
\DeclareMathOperator{\diag}{diag}
\DeclareMathOperator{\Spec}{Spec}
\DeclareMathOperator{\Sing}{Sing}
\DeclareMathOperator{\Sym}{Sym}

\DeclareMathOperator{\depth}{depth}

\def\xto{\xrightarrow}

\newtheorem{theorem}{Theorem}[section]
\newtheorem{lemma}[theorem]{Lemma}
\newtheorem{corollary}[theorem]{Corollary}
\newtheorem{proposition}[theorem]{Proposition}
\theoremstyle{definition}
\newtheorem{remark}[theorem]{Remark}
\newtheorem{remarks}[theorem]{Remarks}
\newtheorem{example}[theorem]{Example}

\newtheorem{definition}[theorem]{Definition}
\newtheorem{question}[theorem]{Question}

\begin{document}

\title{New Free Divisors from Old}

\author[R.-O. Buchweitz]{Ragnar-Olaf Buchweitz}
\address{Dept.\ of Computer and Mathematical Sciences, 
University of Tor\-onto at Scarborough, Tor\-onto, Ont.\ M1A 1C4, Canada}
\email{ragnar@utsc.utoronto.ca}

\author[A.Conca]{Aldo Conca}
\address{Dipartimento di Matematica, Universit\'a di Genova, 
Via Dodecaneso 35, I-16146 Genova, Italia}
\email{conca@dima.unige.it}

\thanks{The first author was partly supported by NSERC grant
  3-642-114-80.}

\date{\today}



\begin{abstract}
We present several methods to construct or identify families of free divisors such as 
those annihilated by many Euler vector fields, including binomial free divisors, or
divisors with triangular discriminant matrix.
We show how to create families of quasihomogeneous free divisors through the chain 
rule or by extending them into the tangent bundle. We also discuss whether general 
divisors can be extended to free ones by adding components and show that adding a
normal crossing divisor to a smooth one will not succeed.
\end{abstract}

\maketitle

{\footnotesize\tableofcontents}

\section{Introduction} 
The goal of this note is to describe some basic operations  that allow to construct new free 
divisors from given ones, and  to classify toric free surfaces and binomial free divisors. We mainly 
deal with weighted homogeneous polynomials over a field of characteristic $0$, though several 
statements and constructions generalize to power series.  

A (formal) {\em free divisor\/} is a reduced polynomial (or power series) $f$ in variables 
$x_1,\dots,x_n$ over a field $K$ such that its Jacobian ideal 
$J(f) = (\tfrac{\partial f}{\partial x_{1}},...,\tfrac{\partial f}{\partial x_{n}})+(f)$ is perfect of 
codimension $2$ in the polynomial or power series ring. 
For generalities about free divisors and their importance in singularity theory 
we refer to, say, \cite{BEvB} and the references therein.  

A  determinantal characterization of free divisors is due to K.~Saito \cite{Sai}: a reduced polynomial 
$f$ is a free divisor if and only if there exists a matrix $A$ of size $n\times n$ with entries in the 
relevant polynomial or power series ring such that $\det(A)=f$ and $(\nabla f) A\equiv 0 \bmod{(f)}$, 
where $\nabla f = (\tfrac{\partial f}{\partial x_{1}},...,\tfrac{\partial f}{\partial x_{n}})$ is the usual 
gradient of $f$. In that case $A$ is called a {\em discriminant\/} (or {\em Saito\/}) {\em matrix\/} 
of the free divisor. 

The normal crossing divisor $f = x_1\cdots x_k$, for some $1\leqslant k\leqslant n$, provides a 
simple example of a free divisor. Indeed, it is an example of a {\em free arrangement}, that is, a 
hyperplane arrangement given by linear equations $\ell_i=0$ such that  the product 
$f=\prod_{i} \ell_i$ is a free divisor, see \cite{OrT} for more on free arrangements.  

Section 2 contains  generalities and notation. In Section 3 we study homogeneous polynomials that
are annihilated by  $n-2$ linearly independent {\em Euler vector fields}, that is, polynomials $f$ such 
that the vector space generated by the linear derivatives $\{x_{i}\partial f/\partial x_{i}\}_{i=1,...,n}$ 
is of dimension at most $2$.  We show that such a polynomial is a free divisor provided the gradient 
$\nabla f$ vanishes as an element of the first homology module of the associated 
{\em Buchsbaum-Rim complex}. As an application, we classify in Theorem \ref{poly3} those 
{\em free surfaces\/} $\{f(x,y,z)=0\}$ that are weighted homogeneous and annihilated by some 
Euler vector field.

In Section 4 we present a {\em composition formula\/} or {\em chain rule\/} for free divisors. Such a 
formula implies, for instance, that if $f$ and $g$ are free divisors in distinct variables then $fg(f+g)$ 
is also a free divisor. 

In Section 5 we exhibit some {\em triangular\/} free divisors, that is, free divisors whose discriminant 
matrix has a triangular form. It follows, for instance, that for natural numbers 
$t\geqslant 1, n\geqslant 2$, the polynomial $\prod_{j=2}^n (x_1^t+\dots+x_j^t)$ is a free 
divisor.

In Section 6 we characterize {\em binomial free divisors\/} by showing that a binomial in $n+2$ 
variables $x_1,\dots,x_n,y,z$ is a free divisor if and only if it is, up to permutation and scaling of the 
variables, of the form 
\[
x_1\cdots x_n y^uz^t\left(y^\alpha\prod x_{i}^{a_i} + z^\beta \prod x_{i}^{b_i}\right)
\] 
with $\min(a_i,b_i)=0$, $\alpha,\beta>0$, and $0\leqslant u,t\leqslant 1$. In particular, any 
reduced binomial is a {\em factor of a free divisor}. 
This observation leads us to ask whether any reduced polynomial is a factor 
of a free divisor. We discuss this question in Section 7, where we show that the simplest approach 
will not work: If $f$ is a smooth form of degree greater than $2$ in more than $2$ variables then 
$x_1\cdots x_nf$ is not a free divisor.

In the final Section 8, we point out that  homogeneous free divisors {\em extend into the tangent 
bundle\/}: along with $f$, the polynomial 
\[
 f(\tfrac{\partial f}{\partial x_{1}}y_{1}+\cdots +\tfrac{\partial f}{\partial x_{n}}y_{n})
\]
in twice as many variables $x_{1},..., x_{n}; y_{1},..., y_{n}$ is again a free divisor. Moreover, it will 
again be {\em linear}, if this holds for $f$.

We want to point out that similar ``extension problems'' for free divisors have been considered by others as well,
especially in \cite{DPi, MS, STo}.

\subsection*{Acknowledgements}
The authors began discussing the results presented here when they met at the
{\sc CIMPA} School on Commutative Algebra,
26 December 2005 to 6 January 2006, in Hanoi, Vietnam.
We want to thank the colleagues who organized that school for the stimulating atmosphere
and generous hospitality.

Special thanks are due to Eleonore Faber who not only produced the pictures included here in Sections 3 and
5, but also provided the (counter-)example in Remark \ref{faber}.
\section{Notation and Generalities} 
Let $R$ be the polynomial ring $K[x_1,\dots,x_n]$ or formal power series ring $K[\![x_1,\dots,x_n]\!]$
over a field $K$ of characteristic $0$. 
Let  $\theta := \theta_{R/K} \cong \oplus_{i=1}^{n}R\partial_{x_{i}}$ denote
the module of  vector fields (or  $K$-linear derivations) of $R$, with 
$\partial_{x_{i}}$ being shorthand for the corresponding partial derivative, 
$\partial_{x_{i}}:= \frac{\partial}{\partial x_{i}}$.
For $f\in R$, we further abbreviate $f_{i}:= f_{x_{i}}:=\partial_{x_{i}}f$, 
so that the {\em gradient\/} of $f$ with respect to the chosen
variables is given by the vector $\nabla f = (f_{1},\dots,f_{n})$ .

\begin{definition}
For $a = (a_{1},\dots,a_{n})\in K^{n}$, we call the linear vector field $E_a=\sum_{i}a_{i}x_{i}
\partial_{x_{i}}$ the {\em Euler vector field\/}  associated to  $a$. 
It is an {\em Euler vector field for $f$}, if $E_a(f)=\delta f$, for some $\delta\in\ZZ$. 

A vector $w\in \ZZ^n$ induces naturally a $\ZZ$-grading on $K[x_1,\dots,x_n]$ 
by setting $\deg_w x_{i}=w_i$. Accordingly, one can assign to any non-zero polynomial $f$ 
a degree $\deg_w(f)$, and that polynomial is {\em $w$--homogeneous}, that is, homogeneous with 
respect to this grading, if all the nonzero monomials in $f$ are of degree $\deg_w(f)$. 
If $f\in R$ is  $w$-homogeneous,  then  $E_w(f)=\deg_w(f)f$.  
\end{definition}

The Jacobian%
\footnote{Some authors; see e.g.~\cite[p.110]{GLS}; call this the {\em Tjurina ideal\/} to distinguish it 
clearly from the ideal generated by just the partial derivatives that describes the {\em critical locus\/} 
of the map defined by $f$.}
ideal $J(f)$ of $f$ is, by definition, $(f_1,\dots,f_n)+(f)\subseteq R$. 
Note that $J(f)=(f_1,\dots,f_n)$ precisely when there exists a derivation $D\in \theta$ such that 
$D(f)=f$. This happens, for example, if $f$ is homogeneous  of non-zero degree with respect  to 
some weight $w\in \ZZ^n$. It is well known that, in general, $J(f)$ defines the {\em singular locus\/}
of the hypersurface ring $R/(f)$, equivalently, the hypersurface $\{f=0\}$ in affine $n$--space 
$\AA^{n}_{K}$. 

\begin{definition} A (formal) {\em free divisor\/} is a polynomial (or power series) $f$, 
whose Jacobian ideal $J(f)$ is {\em perfect\/}%
\footnote{We allow the ideal to be improper, thus, the empty set is perfect of any codimension.
However, the zero ideal is, by convention, not perfect of any codimension, and we always assume $f\neq 0$.}
of codimension $2$ in $R$.

In particular, $f$ is then {\em squarefree}, equivalently, the hypersurface ring $R/(f)$ is 
{\em reduced\/}, --- and we then simply also call $f$ {\em reduced} --- and the singular locus 
of that hypersurface is a Cohen-Macaulay subscheme of codimension two in $\Spec R$.
\end{definition} 

\begin{example}
As simplest examples, any separable polynomial in $K[x]$ defines a free divisor, and so does 
any reduced $f\in K[x,y]$.
\end{example}

K.~Saito, who introduced the notion, gave the following important criterion for  $f$ to be a free divisor: 

\begin{theorem}{\sc (Saito \cite{Sai})}
\label{saito}
Let $f\in R$ be reduced. Then $f$ is a free divisor if and only if there exists a $n\times n$ matrix $A$ 
with entries in $R$ such that $\det A=f$ and $(\nabla f)A\equiv 0 \bmod (f)$. 
\end{theorem}

The matrix $A$ appearing in this criterion is called a {\em discriminant\/} (or {\em Saito\/}) 
{\em matrix\/} of $f$. If the entries of $A$ can be chosen to be linear polynomials, then $f$ is called a
{\em linear\/} free divisor. Note that $f$ is then necessarily a homogeneous polynomial of degree 
$n$.The {\em normal crossing divisor\/} $f = x_{1}\cdots x_{n}$ is a simple example of a linear free 
divisor.

\begin{remark}
It follows immediately from this criterion that a free divisor $f\in R$ remains a free divisor in any
polynomial or power series ring over $R$. When viewed as an element of such larger ring, $f$ is 
called the {\em suspension\/} of the original free divisor from $R$.
\end{remark}

A different way to state the criterion, and to link it with the definition we chose, denote
$\Der(-\log f)\subseteq \theta$ those vector fields $D$ such that $D(f)\in (f)$, equivalently,
$D(\log f)=D(f)/f$ is a well defined element of $R$. With this notation, one has a short exact sequence
of $R$--modules
\begin{align*}
\xymatrix{
0\ar[r]&\Der(-\log f)\ar[r]&\theta\ar[r]^-{df}&J(f)/(f)\ar[r]&0
}
\end{align*}
and a reduced $f$ is a free divisor if, and only if, $\Der(-\log f)$ is a free $R$--module, necessarily 
of rank $n$. A discriminant matrix is then simply the matrix of the inclusion 
$\Der(-\log f)\subseteq \theta$, when bases of these free modules are chosen.

Now we turn to our results.

 \section{Polynomials Annihilated by Many Euler Vector Fields}
 In this section we assume that 
\begin{enumerate}[\rm (a)]
\item $f\in R$ is a nonzero squarefree polynomial that belongs to the ideal of its 
 derivatives, $f\in (f_{1},\dots,f_{n})\subseteq R$.
 \item  The $K$-vector space of Euler vector fields annihilating $f$ has dimension at least $n-2$.  
 In other words, there exist $n-2$ linearly independent  Euler vector fields 
 $E_{j}= \sum_{i}a_{ij}x_{i}\partial_{x_{i}}$, for $j=1,\dots,n-2$, such that $E_{j}(f)=0$. 
 Denote by $A$ the  $n\times(n-2)$ scalar matrix $(a_{ij})$ and by $B$ the matrix $(a_{ij}x_{i})$ of the 
 same size. 
\end{enumerate}
 
Under these assumptions the Jacobian ideal of $f$ is equal to the ideal of its partial derivatives 
and has codimension at least two. To show that it defines a Cohen-Macaulay subscheme of 
codimension two, it suffices thus to find a {\em Hilbert--Burch matrix\/}, 
necessarily of size $n\times(n-1)$, for the partial derivatives. By assumption, we have a matrix 
equation in $R$ of the form
\[
(\nabla f)  B= (0,0,\dots,0)\,.
\]
We need one more syzygy! More precisely; see, for example, \cite[20.4]{Eis}; to get a Hilbert--Burch matrix for 
$(f_{1},\dots,f_{n})$, we want a column vector $w:=(w_{1},\dots,w_{n})^{T}$ with entries from $R$,
such that we have an equality of sequences of elements from $R$ of the form
\[
(f_{1},\dots,f_{n}) = I_{n-1}(C)\,,
\]
where $C$ is obtained from $B$ by appending the column vector $w$, and  $I_{n-1}$ denotes
the sequence of appropriately signed maximal minors of the indicated $n\times(n-1)$ matrix.

Define a $R$--linear map from $R^{n}$ to $R^{n}$ through 
\[
\epsilon(w_{1},\dots,w_{n}) := I_{n-1}( B \mid w)\,,
\]
where $(B\mid w)$ denotes the $n\times(n-1)$--matrix obtained from $B$ by adding the column $w$.
 
Clearly, $B \circ \epsilon=0$, and the sequence of free (graded) $R$--modules
\begin{align*}
\mathbf{BR}(B) \quad\equiv\quad \left(F_{2} = R^{n}(n-1)
\xto{\ \partial_{2}=\epsilon\ } F_{1}=R^{n}(-1)
\xto{\ \partial_{1}=B\ } F_{0}=R^{n-2}\to 0\right)
\end{align*}
is the beginning of the {\em Buchsbaum--Rim complex\/} for the matrix $B$; see, for example, 
\cite[Appendix A.2]{Eis}. 
By the given setup, the vector $\nabla f \in F_{1}$ is a cycle in this complex, and the required
vector $w$ exists if, and only if, the class of $\nabla f$  is zero in the first homology group 
$H_{1}(\mathbf{BR}(B))$ of this Buchsbaum--Rim complex. 

Now, if the ideal of the maximal minors of $B$ has the maximal possible codimension, 
equal to $n-(n-2)+1=3$, then the entire Buchsbaum--Rim complex is exact and so, in particular,
$H_{1}(\mathbf{BR}(B))=0$. 

The minor  of $B$  obtained by deleting rows $i$ and $j$ is the monomial 
$u_{ij}  x_1\dots x_n/x_{i}x_j$, where $u_{ij} $ is the minor of $A$ obtained by deleting the rows 
corresponding to $i$ and $j$. The ideal generated by these minors will have maximal codimension 
if, and only if,  all the maximal minors of $A$  are non-zero. 
 
Summing up, we have the following result.

\begin{proposition} Under the assumption {\em (a)\/} and {\em (b)}, and with the notation as above,
\label{BuRim}

 \begin{enumerate}[\rm (1)]
\item The polynomial $f$ is  a  free divisor  if, and only if, the class of $\nabla f$ in the first homology 
$H_{1}(\mathbf{BR}(B))$ of the Buchsbaum--Rim complex associated to $B$ vanishes.
\item If all the  maximal minors of $A$ are non-zero, then $f$ is a free divisor. 
\end{enumerate}
\end{proposition}

\begin{example}
Consider
\begin{align*}
f = u x^{a} - vx^{b}
\end{align*}
with $u,v\in K$ nonzero and $a,b\in \NN^n$ different exponents with $\min(a_{i},b_{i}) \leqslant 1$, 
for each $i$, to ensure that $f$ is reduced. The Euler vector field 
$\sum_{i=1}^n c_i x_{i}\partial/\partial x_{i}$ then annihilates  $f$ if, and only if, 
$\sum a_ic_i=0$ and $\sum b_ic_i=0$.  Assuming $a_{i}b_{j}-a_{j}b_{i}\neq 0$ for some pair of indices
$i<j$, the space of Euler vector fields annihilating $f$  has dimension $n-2$. The corresponding 
$n\times (n-2)$ coefficient matrix $A$ then satisfies $\binom{a}{b}A=0$, where $\binom{a}{b}$ is the 
obvious $2\times n$ matrix of scalars. Linear algebra tells us that the maximal minors of $A$ are then, 
up to sign and a common non-zero constant, equal to the maximal minors of 
$\binom{a}{b}$.  By virtue of Proposition \ref{BuRim}(2)  we can conclude that if $a_{i}b_{j}-a_{j}b_{i}\neq 0$ for all 
pairs $i<j$, then the binomial $f$ is a free divisor. 
\end{example}

In Section \ref{binocla} below we will give a complete characterization of homogeneous binomial free 
divisors. 

In three variables the  considerations above lead to a complete characterization of free divisors that are 
weighted homogeneous and annihilated by an Euler  vector field. To write down the corresponding 
Hilbert--Burch matrices in a compact form,  the following tool will be useful.

\begin{definition}
Let $d > 0$ be a natural number, $R=K[x_{1},\dots,x_{n}]$ a polynomial ring over a field $K$ of 
characteristic zero, and $ y = \{y_{1},\dots,y_{m}\}$ a subset of the variables $x$. 
Define a $K$--linear endomorphism $(\deg+d)_{ y}^{-1}$ on $R$ through the following action on 
monomials:
\begin{align*}
(\deg+d)_{ y}^{-1}( x^{ e}) := \frac{1}{|e|_{ y}+d} x^{ e}\,,
\end{align*}
where $|e|_{ y} := \sum_{i, x_{i}\in  y}e_{i}$ denotes the usual total degree of $ x^{ e}$
with respect to the variables $ y$.

In words, $(\deg+d)_y^{-1}$ has the polynomials that are homogeneous of total degree $a$ in the 
variables $ y$ as eigenvectors of eigenvalue $1/(a+d)$. If  $y$ is the set of all variables then the  
corresponding $K$--linear endomorphism will be simply denoted by $(\deg+d)^{-1}$. \end{definition}

As is well known, the endomorphism just defined can be used to split in characteristic zero
the tautological Koszul complex on the variables. Here we will use the following form.

\begin{lemma}
Let $V=\oplus_{i}Kx_{i}$ be the indicated vector space over $K$ and 
$V\cong \oplus_{i}K\xi_{i}, x_{i}\mapsto \xi_{i}$ an isomorphic copy of it.
Let $\KK^{\bdot} = \SSS_{K}V\otimes_{K}\Lambda_{K}V\cong R\otimes_{K} \Lambda^{\bdot}_{K}
(\xi_{1},\dots,\xi_{n})$ be the exterior algebra over $R$ on variables $\xi_{i}$, the graded 
$R$--module underlying the usual Koszul complex.

The  $R$--linear derivation $\partial := \sum_{i}a_{i}x_{i}\frac{\partial}{\partial \xi_{i}}$ defines
a differential on $\KK$ for any choice of $a_{i}\in K$. Let $W\subseteq V$ denote
the subspace generated by those variables $ y$ among the $ x$, for which $a_{i}\neq 0$,
and denote by $\eta_{j}$ the corresponding variables among the $\xi_{i}$ in the isomorphic copy of 
$W$. 

If $\omega \in \KK^{m}$ is a {\em cycle\/} for $\partial$, then the class 
of $\omega$ in $H_{i}(\KK^{\bdot},\partial)$ is zero if, and only if, $\omega = 0$ in
$R/(y)\otimes \Lambda^{i}(V/W)$. 
In that case, $\omega':=(\sum_{j}\frac{1}{a_{j}}d\eta_{j}\partial_{y_{j}}) \circ (\deg+d)_{y}^{-1}(\omega)$
provides a boundary, $\partial(\omega') =\omega$. \qed
\end{lemma}

\begin{theorem}
\label{poly3}
Let $K$ be a field of characteristic zero and $f\in K[x,y,z]$ a reduced 
polynomial in three variables such that $f$ is contained in the ideal of 
its partial derivatives, $f\in (f_{x},f_{y},f_{z})$.

Assume further that there is a triple $(a,b,c)$ of elements of $K$ that are not all  zero such that the Euler vector field 
$E=ax\frac{\partial }{\partial x}+by\frac{\partial}{\partial y}+cz\frac{\partial}{\partial z}$ satisfies
$E(f)=0$. 

We then have the following possibilities, up to renaming the variables:
\begin{enumerate}[\rm(1)]
\item If $abc\neq 0$, then $f$ is a free divisor with Hilbert--Burch matrix
\begin{align*}
(f_{x},f_{y},f_{z}) = 
I_{2}\left(\begin{array}{ccc}
\vphantom{\dfrac{1}{2}}ax & & \big(\tfrac{1}{c}- \tfrac{1}{b}\big)(\deg+2)^{-1}(f_{yz})\\
\vphantom{\dfrac{1}{2}}by & &\big(\tfrac{1}{a}- \tfrac{1}{c}\big)(\deg+2)^{-1}(f_{xz}) \\
\vphantom{\dfrac{1}{2}}cz & & \big(\tfrac{1}{b}- \tfrac{1}{a}\big)(\deg+2)^{-1}(f_{xy})
\end{array}\right)
\end{align*}
where $f_{**}$ denotes the corresponding second order derivative of $f$.
\item If $a=0$, but $bc\neq 0$, then $f$ is a free divisor if, and only if, $f_{x}\in (y,z)$.
If that condition is verified and $f_{x}= yg+zh$, then
$f_{y}/cz = -f_{z}/by$ is an element of $R$ and a Hilbert--Burch matrix is given by
\begin{align*}
(f_{x},f_{y},f_{z}) = 
I_{2 }\left(\begin{array}{cc}
0 & f_{y}/cz = -f_{z}/by\\
by & -h/c \\
cz & g/b
\end{array}\right)
\end{align*}
\item If $a=b=0$, then $f$ is independent of $z$ and so, as the suspension of a reduced plane curve, is
a free divisor.
\end{enumerate}
\end{theorem}
\begin{proof}
We simply need to verify that the Hilbert--Burch matrix is correct. 
One may either use now the preceding lemma, or calculate directly, as we will do.
We just verify that, in case (1), the minor obtained when deleting the first row is correct,
leaving the remaining calculations to the interested reader. It suffices to check the case when
$f=x^{e_{1}}y^{e_{2}}z^{e_{3}}$ is a monomial with $ae_{1}+be_{2}+ce_{3} = 0$ and 
$e_{i}\geqslant 0, |e| > 0$. Then,
\begin{align*}
&by(1/b- 1/a)(\deg+2)^{-1}(f_{xy})-cz (1/a- 1/c)(\deg+2)^{-1}(f_{xz}) \\
=\ & by(1/b- 1/a)(\deg+2)^{-1}(e_{1}e_{2}x^{e_{1}-1}y^{e_{2}-1}z^{e_{3}})\\
&\ -
cz (1/a- 1/c)(\deg+2)^{-1}(e_{1}e_{3}x^{e_{1}-1}y^{e_{2}}z^{e_{3}-1})\\
=\ &\frac{ e_{1}e_{2}}{|e|} (1-b/a)x^{e_{1}-1}y^{e_{2}}z^{e_{3}} -
\frac{ e_{1}e_{3}}{|e|} (c/a- 1)x^{e_{1}-1}y^{e_{2}}z^{e_{3}}\\
=\ &f_{x}\left(e_{2}(a-b)-e_{3}(c-a)\right)/a|e|\\
=\ &f_{x}\left((e_{2}+e_{3})a -e_{2}b-e_{3}c)\right)/a|e|\\
=\ &f_{x}
\end{align*}
as required.
\end{proof}

To apply this result, we need to detect Euler  vector fields annihilating given polynomials, and 
the following remark is useful for this purpose. 
\begin{remark} 
Assume $f$ is a polynomial that is homogeneous with respect to  two weights  $w,v\in \ZZ^n$. 
For every $a,b\in \ZZ$, the polynomial $f$ is then homogeneous with respect to $aw+bv$, 
of degree $a\deg_w(f)+b\deg_v(f)$. Taking  $a=\deg_v(f)$ and $b=-\deg_w(f)$, we conclude 
that $f$ is homogeneous of degree $0$ with respect to  $\deg_v(f)w-\deg_w(f)v$, and so the
corresponding Euler vector field annihilates $f$. If further some degree $a\deg_w(f)+b\deg_v(f)$
is not zero, then $f$ satisfies the assumption $(a)$ from the beginning.
\end{remark} 

This remark can be applied as follows.

\begin{example}  Set 
\[
f(x,y,z)=x^{\gamma_1}y^{\gamma_2}z^{\gamma_3} \Pi_{i=1}^k( x^a-\alpha_i y^bz^c)
\]
with   $a,b,c,k\in \NN\setminus\{0\}$,  $\gamma_j\in \{0,1\}$ and $\alpha_i\in K$.
Assume that the $\alpha_i$ are non-zero and distinct  so that $f$ is reduced. 
Then $f$ is a free divisor if, and only if, not both $\gamma_2$ and $\gamma_3$ equal $0$,
equivalently, $\gamma_{2}+\gamma_{3}>0$. 
To prove the statement, take $v=(0,c,-b)$ and $w=(b,a,0)$, so that $f$ becomes homogeneous  
with  respect to both $v$ and $w$, satisfying 
\[
\deg_v(f)=c\gamma_2-b\gamma_3\quad \text{and}\quad 
\deg_w(f)=b\gamma_1+a\gamma_2+kab\neq 0\,.
\]  
Hence, by the remark above, $f\in (f_{x}, f_{y}, f_{z})$, and the Euler vector 
field associated to 
\begin{align*}
\deg_v(f)w-\deg_w(f)v &= (c\gamma_2-b\gamma_3)(b,a,0) - (b\gamma_1+a\gamma_2+kab)(0,c,-b)\\
&=-b(-c\gamma_2+b\gamma_3,a\gamma_3+c\gamma_1+kac, -b\gamma_1-a\gamma_2-kab)
\end{align*}
annihilates $f$. Clearly, 
the second and the third coordinates of this vector  are non-zero, while the first one equals
$b(c\gamma_2-b\gamma_3)$. Now, if  $\gamma_2$ or  $\gamma_3$ is non-zero, then $f_x\in (y,z)$ 
and we conclude by Theorem \ref{poly3}, either part  (1)  or (2), that $f$ is a free divisor. 

On the other hand, if  $ \gamma_2=\gamma_3=0$ then $f$ contains a pure power of $x$ and so 
$f_x\not\in (y,z)$. We may then conclude by Theorem \ref{poly3}(2)  that $f$ is not a free divisor.
\end{example} 

\begin{remark}
Some isolated members of this family of examples have been identified as free divisors before:
\begin{align*}
f = y(x^{2}-yz)\quad\text{or}\quad f = xy(x^{2}-yz)\,,
\end{align*}
the quadratic cone with, respectively, one or two planes, of which one is tangent, or
\begin{align*}
f = y(x^{2}-y^{2}z)\,,
\end{align*}
the Whitney umbrella with an adjoint plane; see \cite{MS}.

A remarkable feature of this example is that it exhibits free surfaces with arbitrarily many
irreducible components that are not suspended, in that we can, for example, extend the family of 
examples involving quadratic cones to
\[
f=x^{\gamma_{1}}y^{\gamma_{2}}z^{\gamma_{3}}\prod_{i=1}^{k}(x^{2}-\alpha_{i}yz)
\]
for $k\geqslant 1, \gamma_{j}\in \{0,1\}$ with $\gamma_{2}+\gamma_{3}\neq 0$ and
scalars $\alpha_{i}\in K$ satisfying $\prod_{i=1}^{k}\alpha_{i}\prod_{i<j}(\alpha_{i}-\alpha_{j})\neq 0$.
Such $f$ will clearly have $\gamma_{1}+\gamma_{2}+\gamma_{3}+k$ many irreducible components,
$1\leqslant \gamma_{1}+\gamma_{2}+\gamma_{3}\leqslant 3$ among them planes.
\begin{figure}[!h]
\begin{tabular}{c@{\hspace{1.5cm}}c}
\includegraphics[width=0.39 \textwidth]{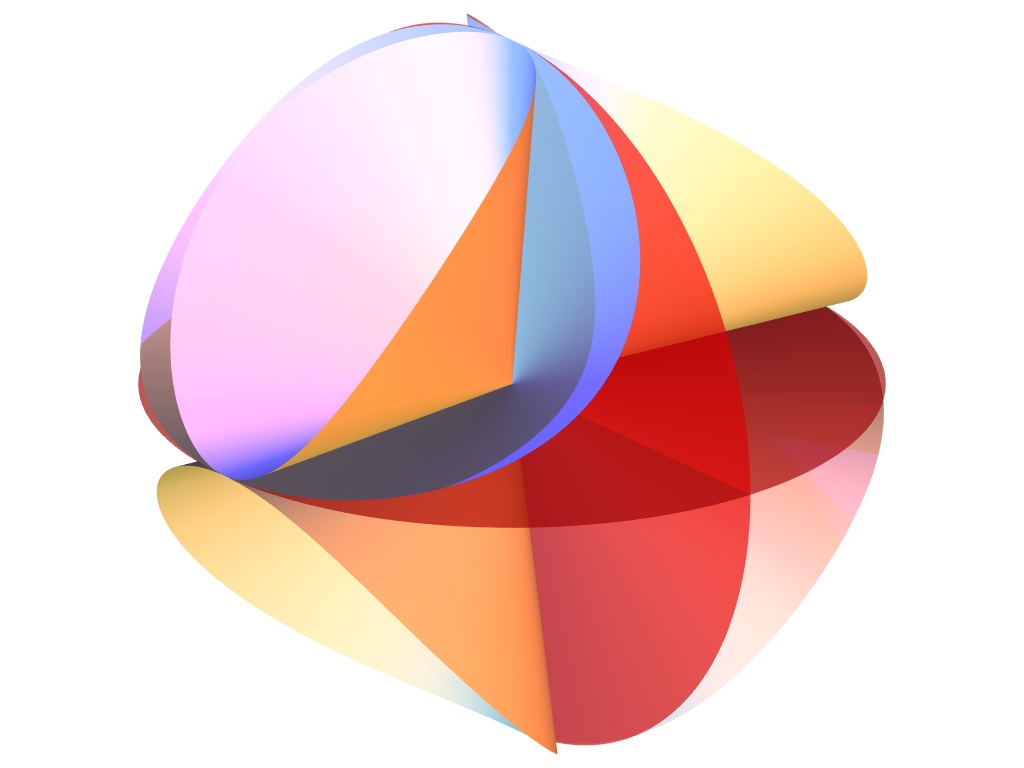} & 
\includegraphics[width=0.39 \textwidth]{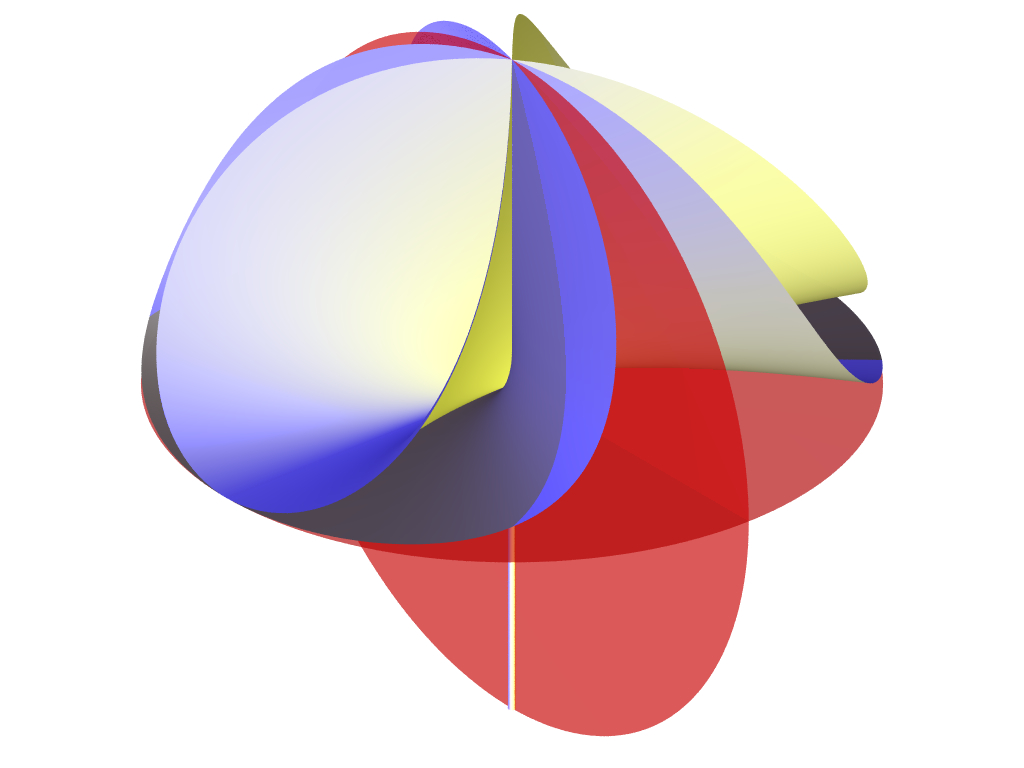}
\end{tabular}
\caption{ \label{fig:binomial} The free divisors defined by $h=yz(x^2-5yz)(x^2-\frac{1}{2}yz)(x^2+yz)$ (left) 
and $h=yz(x^2-\frac{1}{2}y^2z)(x^2+5y^2z)$ (right).}
\end{figure}
\end{remark}

\section{A Chain Rule for Quasihomogeneous Free Divisors}
We start with a simple observation:  if $f\in K[x]=K[x_1,\dots,x_n]$ and $g\in K[y]=K[y_1,\dots,y_m]$ 
are free divisors then $fg\in K[x,y]$ is a free divisor. To see this, one just takes the 
discriminant matrices $A,B$ associated to $f$ and $g$, and notes that the block matrix 
\[
\left(\begin{array}{cc}
A  & 0  \\
0 & B
\end{array}\right) 
\]
is a discriminant matrix  for $fg$ that one can think of as the pullback of the planar normal crossing 
divisor along the map with components $(f,g)$. 
Such free divisors have been called ``product-unions'' by J.~Damon 
\cite{Dam} or ``splayed'' divisors by Aluffi and E.~Faber \cite{AFa}.

If $f=f_{1}\cdots f_{k}$ is square free, then a vector field $D$ is logarithmic for $f$ if, and only if,
$D$ is logarithmic for each $f_{i}$, as
\begin{align*}
D(\log f) =\sum_{i}D(\log f_{i}) =\sum_{i}\frac{D(f_{i})}{f_{i}}
\end{align*}
can only be an element of $R$ if that holds for the summands.

We now use these observations to establish a {\em chain rule\/} for free divisors.
In this form, the result and its proof are due to Mond and Schulze \cite[Thm.4.1]{MS}, while 
we originally had obtained a weaker result. We include an algebraic version of the 
proof, and strengthen their result by removing the hypothesis that no $f_{i}$ be a smooth divisor.

\begin{theorem} 
\label{chainrule}
Let $k\geqslant 1$ be an integer, $K$ a field of characteristic zero.
Assume given a free divisor $f=f_{1}\cdots f_{k}\in R = K[x_{1},\dots,x_{n}]$ that admits
vector fields $E_{j}$, for $j=1,\dots,k$, satisfying $E_{j}(f_{i})=\delta_{ij}f_{i}$, where $\delta_{ij}$ 
is the Kronecker delta.

If $H=y_{1}\cdots y_{k}H_{1}\in Q:=K[y_{1},\dots,y_{k}]$ is a free divisor 
such that $f$ and $H_{1}(f_{1},..., f_{k})$ are without common factor, then the polynomial 
$\tilde H:= H(f_{1},\dots,f_{k})\in R$ is a free divisor.
\end{theorem}

\begin{proof} Because $f$ is a free divisor, its $R$--module of logarithmic vector fields 
$\Der(-\log f)$ is free. It contains the vector fields $E_{i}$, because $E_{i}(f)= f$ by the 
product rule. 
Further, the $E_{i}$ are linearly independent over $R$, as $0 = \sum_{i=1}^{k}g_{i}E_{i} \in\theta$ 
implies $0=\sum_{i=1}^{k}g_{i}E_{i}(f_{j}) = g_{j}f_{j}$, and so $g_{j}=0$ for each $j$. 
In this way, $\oplus_{i=1}^{k}RE_{i}$ becomes a free submodule of $\Der(-\log f)$.

Now any $D\in \Der(-\log f)$ is logarithmic for each 
$f_{i}$ as those elements of $R$ are relatively prime, $f$ being squarefree. 
Therefore, $D\mapsto \sum_{i=1}^{k}D(\log f_{i})E_{i}$ provides
an $R$--linear map $\Der(-\log f)\to \oplus_{i=1}^{k}R E_{i}$ that splits the inclusion, 
and whose kernel consists of those derivations $D$ that satisfy $D(f_{i})=0$ for each $i$.

Therefore, we can extend the $E_{i}$ to a basis $(E_{1}, ..., E_{k}, D_{1}, ..., D_{n-k})$ of $\Der(-\log f)$ 
as $R$--module, with $D_{j}(f_{i})=0$ for $i=1,...,k$ and $j=1,...,n-k$. 

Let $C$ be the $n\times n$ matrix over $R$ that expresses the just chosen basis of $\Der(-\log f)$ 
in terms of the partial derivatives $\tfrac{\partial}{\partial x_{j}}$, for $j=1,...,n$, so that
\begin{align*}
(E_{1},..., E_{k}, D_{1}, ...,D_{n-k})= (\tfrac{\partial}{\partial x_{1}},...,\tfrac{\partial}{\partial x_{n}})C\,.
\end{align*}
The matrix $C$ is then a discriminant matrix for $f$, and, in particular, $\det C =f$.

Now we turn to $H\in Q$ and observe that any $D\in\Der_{Q}(-\log H)$, a logarithmic derivation for 
$H$ over $Q$, is necessarily of the form $D=\sum_{r=1}^{k}y_{r}b_{r}\frac{\partial}{\partial y_{r}}$
for suitable elements $b_{r}\in Q$, as $H$ contains by assumption $y_{1}\cdots y_{k}$ as a factor, 
whence $D(\log y_{r})= b_{r}$ must be in $Q$. In matrix form, a discriminant matrix for $H$ can be
factored as 
\begin{align*}
A:= \diag(y_{1},...,y_{k})B\,,
\end{align*}
where the first factor is the diagonal matrix with entries $y_{r}$ and $B=(b_{rs})$ is a 
$k\times k$ matrix over $Q$ so that the vector fields $\sum_{r}y_{r}b_{rs} \frac{\partial}{\partial y_{r}}$ 
form a $Q$--basis of $\Der_{Q}(-\log H)$. Because $\det A=H$ by Saito's criterion in Theorem
\ref{saito}, it follows that $\det B = H_{1}\in Q$.

Next note that the given $f_{i}$ define a substitution homomorphism $Q\to R$ that 
sends $y_{i}\mapsto f_{i}$. For any $b\in Q$, we denote $\tilde b=b(f_{1},..., f_{k})$ its image in 
$R$. We claim that a derivation $\tilde D := \sum_{r}\tilde b_{r}E_{r}$ is logarithmic for 
$\tilde H\in R$, if $D:= \sum_{r}y_{r}b_{r} \frac{\partial}{\partial y_{r}}$ is logarithmic for $H\in Q$. In 
fact, the usual chain rule for derivations yields first
\begin{align*}
\tilde D(\tilde H) &= \sum_{r=1}^{k}\tilde b_{r}E_{r}(\tilde H)\\
&= \sum_{r=1}^{k}\tilde b_{r}\sum_{s=1}^{k}\widetilde{\frac{\partial H}{\partial y_{s}}} E_{r}(f_{s})\\
&= \sum_{r=1}^{k}f_{r}\tilde b_{r}\widetilde{\frac{\partial H}{\partial y_{r}}} 
\end{align*}
as $E_{r}(f_{s})=\delta_{rs}f_{r}$ by assumption. Now the last term equals $\widetilde{D(H)}$, 
the image of $D(H)$ under substitution. Thus, if $D(H)$ is in $(H)\subseteq Q$, its image is in 
$(\tilde H)\subseteq R$, and so $\tilde D$ is indeed logarithmic for $\tilde H$.

On the other hand, if $D$ is a derivation on $R$ that vanishes on each $f_{i}$, then applying the 
chain rule yet again shows
\begin{align*}
D(\tilde H) &=\sum_{r=1}^{k}\widetilde{\left(\frac{\partial H}{\partial y_{r}}\right)} D(f_{r}) =0\,,
\end{align*}
whence such $D$ is in particular logarithmic for $\tilde H$. Putting everything together,
\begin{align*}
\left(\tfrac{\partial}{\partial x_{1}},...,\tfrac{\partial}{\partial x_{n}}\right)C
\left(\begin{matrix}
\tilde B&0\\
0&I_{n-k}
\end{matrix}
\right)\,,
\end{align*}
with $I_{n-k}$ the identity matrix of indicated size, represents $n$ logarithmic vector fields for 
$\tilde H$. Taking determinants, we get 
\begin{align*}
\det\left(C
\left(\begin{matrix}
\tilde B&0\\
0&I_{n-k}
\end{matrix}
\right)\right)
 = \det C\det \tilde B = \det C\widetilde{\det B} = 
f_{1}\cdots f_{k}\widetilde{H_{1}}=\tilde H\,.
\end{align*} 
Thus, the proof will be completed by Saito's criterion Theorem \ref{saito}, once we show that 
$\widetilde{H_{1}}$ is squarefree, as by assumption $f$ is already squarefree and relatively 
prime to $\tilde H_{1}$. To this end, we use the Jacobi criterion; see e.g. \cite[30.3]{Mat}.
The rank of the Jacobi matrix 
\[
\left( \frac{\partial f_{i}}{\partial x_{j}}\right)_{j=1,...,n}^{i=1,...,k}
\] 
is $k$ outside of $\{f=0\}$, as $E_{1}(f_{1})\cdots E_{k}(f_{k})=f$ is in the ideal of maximal minors
of that matrix. Therefore, $R$ is smooth over $Q$ outside of $\{f=0\}$, and the inverse image 
$\{\tilde H_{1}=0\}$ of $\{H_{1}=0\}$ remains thus reduced.
\end{proof}

We mention the following special case of Theorem \ref{chainrule} as an example.

\begin{corollary}
\label{corofg}   If $f\in K[x]=K[x_1,\dots,x_n]$ and $g\in K[y]=K[y_1,\dots,y_m]$ are 
free divisors that are weighted homogeneous, then $fg(f+g)\in K[x,y]$ is a free divisor. 
\qed
\end{corollary}

\begin{remark}
\label{faber}
In the original treatment of Theorem \ref{chainrule} in \cite{MS}, the hypothesis that $f$ and 
$H_{1}(f_{1},..., f_{k})$ are without common factor is missing. 
That hypothesis is, however, necessary, as is shown by the following example that Eleonore Faber 
kindly provided.

Take $f_1= (1+u)(x^2-y^3), f_2=(1+v)(y^2-x^3)$, and $f_3=(1+w)(f_1^3+f_2^2)$ in $R=K[x,y,u,v,w]$. 
A calculation in {\sc Singular\/} shows readily that $f=f_1 f_2 f_3$ is a free
divisor. The vector fields $E_{1}=(1+u)\partial/\partial u, 
E_{2}=(1+v)\partial/\partial v$, and $E_{3}=(1+w)\partial/\partial w$ certainly satisfy $E_{i}(f_{j})=\delta_{ij}f_{i}$.

Now take $H(y_1,y_2,y_3)=y_1y_2y_3(y_1^3+y_2^2)$, a binomial free divisor according to Theorem
\ref{binofree} below, and observe that 
\[H(f_1,f_2,f_3)=f_1f_2f_3(f_1^3+f_2^2)=f_1f_2(1+w)(f_1^3+f_2^2)^2
\] 
is not reduced, thus, is not a free divisor, as $f$ and $H_{1}(f_{1},f_{2},f_{3})$ have the factor $f_1^3+f_2^2$ 
in common.
\end{remark}

 \section{Triangular Free Divisors}
 
Let $K$ be   a field of characteristic zero. Assume given a ``seed''
$F_{0}\in R:= K[y_{1},\dots,y_{n}]$ and define inductively for $i > 0$ polynomials
\begin{align*}
F_{i} := \alpha_{i}x_{i}^{a_{i}} + \beta_{i}F_{i-1}^{b_{i}}\in Q:= R[x_{1},\dots,x_{i}]
\end{align*}
for natural numbers $a_{i},b_{i} > 0$ and $\alpha_{i},\beta_{i}\in K$ with $\alpha_{i}\neq 0$.

\begin{proposition}
\label{trianfree}
Assume $F_{0}$ is a free divisor in $R$ with discriminant $(n\times n)$--matrix $A$
over $R$. If $F:= F_{i}F_{i-1}\cdots F_{0}$ is reduced, then it is a free divisor over $Q$ with
``triangular'' discriminant matrix of the form
\begin{align*}
B =
\left(\begin{matrix}
A & 0&0&\cdots&0 \\
* & F_{1} & 0 &\cdots&0\\
\vdots & \vdots & \ddots &\ddots&\vdots \\
* & * & * & F_{i-1}&0\\
* & * & * & * & F_i
\end{matrix}\right)
\end{align*}
where the  entries marked ``$*$'' represent elements of $Q$ that can be calculated explicitly.
\end{proposition}

\begin{proof}
First observe that the determinant of the displayed matrix certainly equals $F$.
It thus remains to prove that the we can choose the columns to represent logarithmic vector fields for it.

The proof proceeds by induction on $i\geqslant 0$, the case $i=0$ being true by assumption.
For $i\geqslant 1$, set $G = F/F_{i}$ and assume that the result is correct for $G$.
The last column in $B$ represents the vector field $D=F_{i}{\partial}/{\partial x_{i}}$ and 
we show now that it is  a logarithmic vector field for $F$, that is, $F$ divides $D(F)$:
\begin{align*}
D(F) &= D(F_{i})G= F_{i}\frac{\partial F_{i}}{\partial x_{i}} G = \left(\frac{\partial F_{i}}{\partial x_{i}}\right) F\,,
\end{align*}
the first equality due to the fact that $G$ is independent of $x_{i}$.

To finish the proof, it suffices now to establish the following:

\begin{lemma}\label{trilemma}
Let $D$ be a logarithmic vector field for $G$ as an element of $R[x_{1},\dots,x_{i-1}]$.
\begin{enumerate}[\rm (1)]
\item $D$ is a logarithmic vector field for each factor $F_{0},\dots,F_{i-1}$ of $G$,
so that $c_{F_{j}}:= D(F_{j})/F_{j}\in R[x_{1},\dots,x_{i-1}]$ for each $j=0,\dots,i-1$.
\item The vector field 
\begin{align*}
\tilde D = \frac{b_{i}c_{F_{i-1}}}{\alpha_{i}a_{i}} x_{i}\frac{\partial}{\partial x_{i}} + D
\end{align*}
is the unique extension of $D$ to a logarithmic vector field for $F$ in $Q$. It satisfies
\begin{align*}
\tilde D(F) = \big((b_{i}+1)c_{F_{i-1}}+\sum_{j=0}^{i-2}c_{F_{j}}\big)F\,.
\end{align*}

\end{enumerate}
\end{lemma}

\begin{proof}
The first part was already pointed out above:
if $D$ is any logarithmic vector field for a product $fg$ of coprime 
factors, then it is necessarily a logarithmic vector field for each factor. 

Now we turn to the derivation $D$ given in the statement.
Assume there is an extension $\tilde D = u\frac{\partial}{\partial x_{i}} + D$ 
of $D$ to a logarithmic vector field for $F$. We then get first from the product rule
\begin{align*}
\tilde D(F) &= \tilde D(F_{i})G + F_{i}\tilde D(G)\,,
\intertext{and by definition of $\tilde D$ and $F_{i}$ this evaluates to}
&= \left(u\alpha_{i}a_{i}x_{i}^{a_{i}-1} + \beta_{i}b_{i}F_{i-1}^{b_{i}-1}D(F_{i-1})\right)G + F_{i}D(G)
\intertext{as $\tilde D(H)=D(H)$ for $H$ equal to either $F_{i-1}$ or $G$,}
&=  \left(u\alpha_{i}a_{i}x_{i}^{a_{i}-1} + \beta_{i}b_{i}c_{F_{i-1}}F^{b_{i}}_{i-1}\right)G + c_{G}F_{i}G
\end{align*}
as $D$ is respectively logarithmic for $F_{i-1}$ and for $G$ with the indicated multipliers.

Due to $F = F_{i}G$, we see that $\tilde D(F)$ will be a multiple of $F$ if, and only if, 
$F_{i}= \alpha_{i}x_{i}^{a_{i}} + \beta_{i}F_{i-1}^{b_{i}}$ divides $u\alpha_{i}a_{i}x_{i}^{a_{i}-1} + \beta_{i}b_{i}c_{F_{i-1}}F^{b_{i}}_{i-1}$, if, and only if,
\begin{align*}
u = b_{i}c_{F_{i-1}}x_{i}/a_{i}\,,
\end{align*}
and in that case
\begin{align*}
\tilde D(F) = (b_{i}c_{F_{i-1}}+c_{G})F\,.
\end{align*}
It follows that 
\begin{align*}
\tilde D := \frac{b_{i}c_{F_{i-1}}}{a_{i}} x_{i}\frac{\partial}{\partial x_{i}} + D
\end{align*}
is the unique extension of $D$ to a logarithmic vector field for $F$ as claimed.
Finally, observe that the multiplier in question is
\begin{align*}
c &:= \frac{\tilde D(F)}{F} = b_{i}c_{F_{i-1}} + c_{G} \\
&= b_{i}c_{F_{i-1}}  + \sum_{j=0}^{i-1}c_{F_{j}}\\
&= (b_{i}+1)c_{F_{i-1}}  + \sum_{j=0}^{i-2}c_{F_{j}}
\end{align*}
and that finishes the proof.
\end{proof}
To end the proof of Proposition \ref{trianfree}, if the result holds for $i-1$, we extend the column 
that represents the logarithmic vector field $D$  for $G=F_{i-1}\cdots F_{0}$ in the
displayed discriminant matrix by adding the corresponding 
coefficient  $\frac{b_{i}c_{F_{i-1}}}{a_{i}} x_{i}$ of $\partial/\partial x_{i}$ in $\tilde D$ as the entry 
in the last row of the discriminant matrix for $F$.
\end{proof}
Note that in Proposition \ref{trianfree}  we may take as seed $F_{0}$ any reduced polynomial in two 
variables. 
\begin{figure}[!h]
\begin{tabular}{c@{\hspace{1.5cm}}c}
\includegraphics[width=0.39 \textwidth]{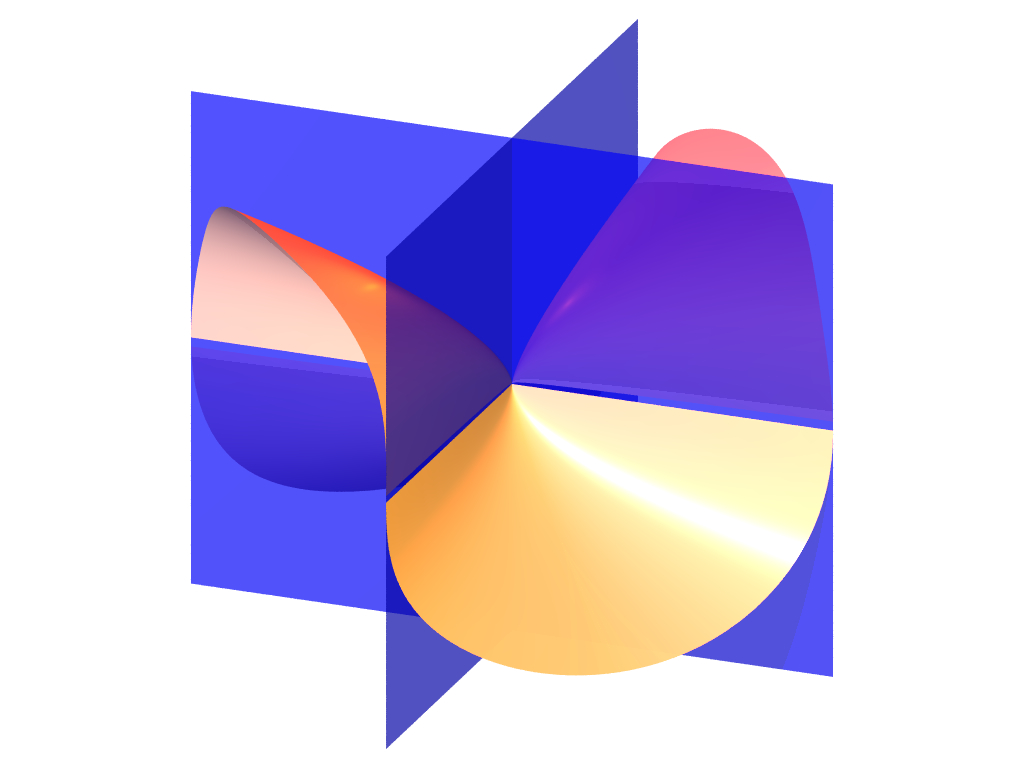} & 
\includegraphics[width=0.39 \textwidth]{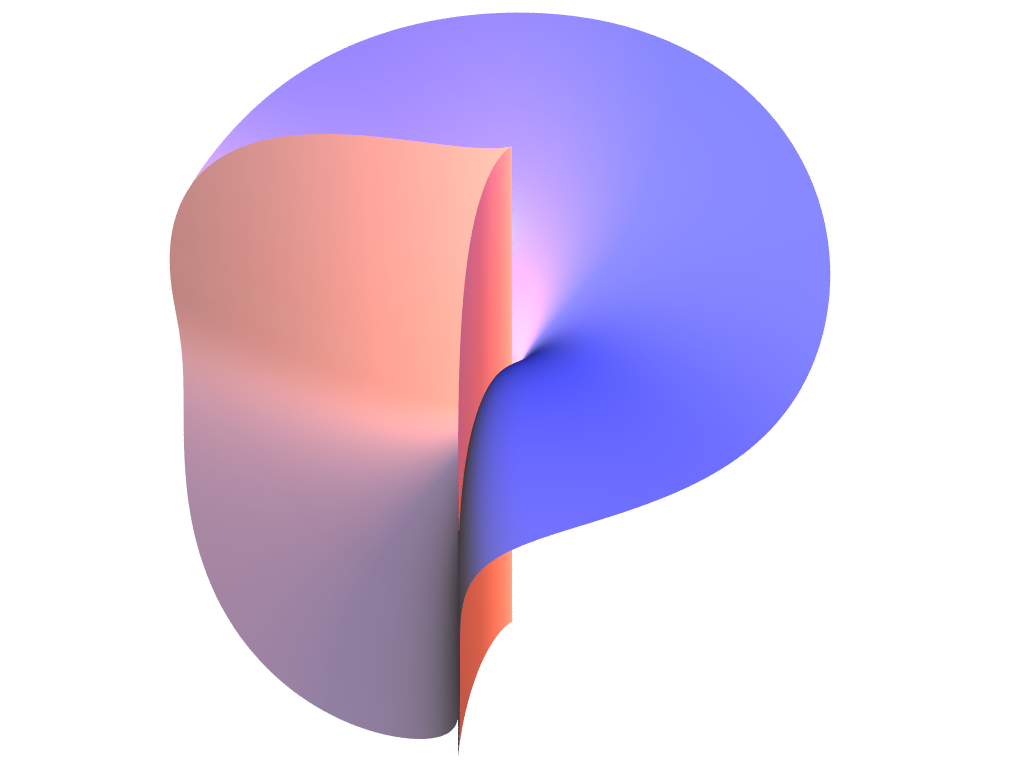}
\end{tabular}
\caption{ \label{fig:triang} The union of a cylinder over an $A_1$-curve and an $A_2$-surface given by 
$h=(x^2-y^2)(x^2-y^2+z^3)$ (left) and the union of a cylinder over an $A_2$-curve and  an $E_8$-surface 
given by $h=(x^2+y^3)(x^2+y^3-z^5)$ (right).}
\end{figure}

\begin{example}
\label{sumpow}
Given positive integers $t_{1},..., t_{i}$, for $j=2,\dots,i$, set $G_j=x_{1}^{t_{1}}+\cdots +x_{j}^{t_{j}}$.  
Take $F_{0}=G_2$  as a seed and set $a_{j}=t_{j+2}, b_{j}=\alpha_{j}=\beta_{j}=1$ to obtain 
$F_{j}=G_{j+2}$ for $j=0,...,i-2$. The resulting product  $G =G_2\cdots G_{i}$ of Brieskorn--Pham 
polynomials is a free divisor by Proposition \ref{trianfree}. 

One can easily calculate the entries of the discriminant matrix. 
To illustrate, we treat the case where each exponent is $t=2$, so that 
$G_j=x_{1}^{2}+\cdots +x_{j}^{2}$.

The first column can be taken as representing the usual Euler vector field that is the unique 
extension of the Euler vector field for $G_{2}$. The second column can be taken to correspond to 
the vector field $D = -x_{2}\partial/\partial x_{1} + x_{1}\partial/\partial x_{2}$ that in turn corresponds 
to the automorphism interchanging $x_{1}$ and $x_{2}$. As for this $D$ one has $D(G_{2}) = 0$, 
Lemma \ref{trilemma}  shows that the corresponding matrix entries below the second row will be 
zero as well.

Now we indicate how to obtain the entries of columns $3$ through $i$.
Counting from the top, start with $D = G_{j}\partial/\partial x_{j}$, thus, putting $G_{j}$ as the entry 
in the $j^{th}$ row as first nonzero entry in column $j\geqslant 3$, and note that 
$D(G_{j}) = 2 x_{j}G_{j}$, so that $c_{G_{j}}= 2x_{j}$. By Lemma \ref{trilemma}, 
the entry below it will be
\begin{align*}
a_{j+1,j} =  \frac{b_{j+1}c_{G_{j}}}{a_{j+1}} x_{j+1} =
 \frac{c_{G_{j}}}{2} x_{j+1} =  x_{j}x_{j+1}
\end{align*}
Now $c_{G_{j+1}}= 2 x_{j}$ again, and induction shows that a relevant discriminant matrix can be taken in the form
\begin{align*}
B = 
\left(\begin{array}{cccccc}
x_{1} & -x_{2} & 0 & 0  &\cdots & 0 \\
x_{2} & x_{1} & 0 &   0  &\cdots & 0 \\
x_{3} & 0 & G_{3} &  0 &\cdots &0 \\
x_{4} & 0 & x_{3}x_{4} & G_{4} & \ddots &  \vdots\\
\vdots & \vdots & \vdots & \vdots & \ddots & 0 \\x_{i} & 0 & x_{3}x_{i} & x_{4}x_{i} &\cdots   & G_{i} \\  &   &   &   &   &  \end{array}\right)
\end{align*}

\end{example}

\section{Binomial Free Divisors}\label{binocla} 

The goal of this section is to investigate binomials $(ux^{a}+ vx^{b})x^{c}$, with $u,v\in K, uv\neq 0$, 
and exponent vectors $a,b,c$ with $|a|,|b|\geqslant 1$,  $\min(a_{i},b_{i})=0$, that are free divisors. 
This forces each entry of $c$ to be in $\{0,1\}$ and we can absorb the constants $u,v$ into the 
variables to reduce to the form $F=L(M+N)$, where $L$ is a product of distinct variables and $M,N$ 
are coprime monomials. 

We further assume $R=K[x_1,\dots,x_{n+2}]$, with $K$ as usual a field of characteristic $0$, 
and we may suppose that $F$ involves all the variables, as otherwise it is just a suspension of a 
divisor that satisfies this requirement.

With these preparations we show the following result.
\begin{theorem} 
\label{binofree}
The binomial  $F=L(M+N)$ as above is a free divisor if 
\begin{enumerate}[\rm (a)]
\item at most one of the variables appearing in $M$ does not appear  in $L$, and
\item at most one of the variables appearing in $N$ does not appear  in $L$.
\end{enumerate}
Note that if $F$ is required to involve all variables, then these conditions imply $\deg L\geqslant n$.

If $F$ is a {\em homogeneous\/} binomial, that is, $\deg M=\deg N$, then the preceding sufficient conditions are also necessary. 
\end{theorem} 

\begin{proof}  For the first claim, we can write, up to a permutation of the variables and 
setting $y=x_{n+1}$ and $z=x_{n+2}$,
\[
F=x_1\cdots x_n y^uz^t G
\]
where 
\[
G=x^ay^\alpha+x^b z^\beta
\]
and  $a,b\in  {\bf N}^n$  with $\min(a_i,b_i)=0$, $\alpha,\beta>0$ and $u,t\in \{0,1\}$. Let $V$ be the 
$K$ vector space generated by the monomials $x_1\cdots x_n x^{a}y^{u+\alpha}z^t$ and
$x_1\cdots x_n x^{a}y^{u}z^{t+\beta}$ involved in $F$. Obviously, $V$ is $2$-dimensional,
the elements $F, zF_z$ form a basis, and $V$ contains $x_{i}F_{x_{i}}$  for each $i=1,\dots,n+2$. 
So we get relations
\begin{align}
&x_{i}F_{x_{i}} +v_i zF_{z}\equiv 0  \bmod (F)\,,
\intertext{
with some $v_i\in K$, for $i=1,\dots n$.  
Now note that}
&F_y=x_1\cdots x_nz^t(uG+\alpha x^ay^{\alpha-1+u}) 
\intertext{and}
&F_z =x_1\cdots x_ny^u(tG+\beta  x^bz^{\beta-1+t})
\end{align}
whence we also get relations
\begin{align}
\beta  yF_{y}&+\alpha    zF_{z}\equiv 0  \bmod (F)  
\intertext{and}
-y^u(tG+\beta x^bz^{\beta-1+t}) F_{y} &+ z^t(uG+\alpha x^ay^{\alpha-1+u}) F_{z}=0\,.
\end{align}
%
%
%
%
Collecting this information in the $(n+2)\times (n+2)$ matrix
\[
A=\left(\begin{array}{cccccc}
 x_1  & 0     &      &  \dots & 0 &  0 \\
 0     &  x_2  & 0 &  \dots & 0 & 0  \\
  \vdots  & & & & & \vdots \\
 0     &  0  &  \dots  &  x_{n} & 0 & 0 \\
 0     &  0  &  \dots       &  0          &  \beta  y  &\ \ \  -y^u(tG+\beta x^bz^{\beta-1+t})  \\
 v_1z&  v_2z& \dots  &   v_nz  &     \alpha z& \ \  \   z^t(uG+\alpha x^ay^{\alpha-1+u}) \end{array}\right)
\]
it follows from (1) and (4) that the first $n+1$ entries of $(\nabla F) A$ are congruent to $0$ modulo 
$F$, while (5) implies that the last entry of  $(\nabla F) A$ equals $0$ already in $R$. 
Finally, it is straightforward that 
\[
\det A=(\beta\alpha+u\beta+t\alpha)F\quad\text{and}\quad \beta\alpha+u\beta+t\alpha \neq 0\,,
\]
whence we conclude from Saito's criterion in Theorem \ref{saito}  that $F$ is a free divisor. 

Next we show that if  $F$ is a {\em homogeneous\/} free divisor  then  conditions  (a), (b) are satisfied.
We argue by contradiction. Suppose that $F$ is a free divisor that involves all variables, but
fails one of the  conditions  (a) or (b). By symmetry, and after permutating the variables, we may
assume that $F$ is of the form: 
\[
F=x^ay^\alpha z^\beta+x^b\,,
\]
where we set $y=x_{n+1}, z=x_{n+2}$ as before, and $a,b\in \NN^n, \alpha>0, \beta>0$. 
With $J$ again the Jacobian ideal of $F$, note that $(y,z) \subseteq (J:x^a y^{\alpha-1} z^{\beta-1})$.
Since $J$ is perfect of  codimension $2$, either $(y,z)$ is a minimal prime of $J$ or 
$ x^a y^{\alpha-1} z^{\beta-1}\in J$. In the former case, $F\in J\subset (y,z)$ implies $x^b\in (y,z)$, and 
that is impossible. In the latter case, 
\[
x^a y^{\alpha-1} z^{\beta-1}\in J \subseteq 
(y^{\alpha-1} z^{\beta}, y^{\alpha} z^{\beta-1})+( \partial x^b/\partial x_{i}\ ; i=1,\dots, n)\,,
\]
and so $x^a y^{\alpha-1} z^{\beta-1}$ must be divisible by $\partial x^b/\partial x_{i}$ for some $i$. 
This contradicts the homogeneity of $F$. 
\end{proof}  

\begin{example}
A particular case of Theorem \ref{binofree} has recently been presented independently by Simis and
Tohaneanu \cite[Prop. 2.11]{STo}:

In our notation from the proof above, they take a homogeneous binomial of the form
$G=x^{a}y^{\alpha}+z^{\beta}$, with $\alpha > 0, |a|+\alpha=\beta$, and $a_{i}\neq 0$ for $i=2,...,n$ 
in $x^{a}=x_{1}^{a_{1}}\cdots x_{n}^{a_{n}}$, so that $G$ is homogeneous of degree $\beta$ and
the only potentially missing variable in the first summand is $x_{1}$. The authors then affirm
that
\begin{align*}
F&= x_{1}\cdots x_{n}(x^{a}y^{\alpha}+z^{\beta})&& \text{and}\\
F &= \frac{x_{1}\cdots x_{n}}{x_{i}}y(x^{a}y^{\alpha}+z^{\beta})&&\text{for some $i=1,...,n$,}
\end{align*}
are homogeneous free divisors. Theorem \ref{binofree} shows that in each case, $zF$ is a 
homogeneous free divisor as well.
\end{example}

\section{``Divisors" of Free Divisors}
The results of the previous sections show that:
\begin{enumerate}[\rm (1)]
\item Any reduced homogeneous binomial has a multiple that is a free divisor by Theorem 
\ref{binofree}. 
\item If $K$ is algebraically closed, then any quadric $Q$ can be put in standard form   
$x_1^2+\dots+x_{i}^2$. Hence it has a multiple that is a free divisor by Example \ref{sumpow}. 
\item If $f,g$ are free divisors in distinct sets of variables, then $f+g$ divides the free divisor $fg(f+g)$
by Corollary \ref{corofg}.
\end{enumerate} 

So we are led to ask:
\begin{question} Let $f$ be a (homogeneous) reduced polynomial. Does there exist a free divisor $g$ 
such that $f$ divides $g$?
\end{question} 
This question is also raised and adressed in \cite{DPi, MS, STo}. 

In light of the discussion above, the first  case to look at is that of cubics in $3$ variables. 
Again, by Example \ref{sumpow}, we know that the Fermat cubic $x^3+y^3+z^3$  divides the free divisor 
$(x^3+y^3)(x^3+y^3+z^3)$. So, what about other smooth cubics or smooth hypersurfaces in general? 
What we can prove is a negative result:  it asserts that a smooth form, in $n>2$ variables of degree 
larger than 2, times a product of $n$ linearly independent linear forms is never a free divisor. 

\begin{theorem} 
\label{divdiv1}  Let   $f$  be  a smooth form of degree $k=\deg f>2$ in $n>2$ variables and  
$\ell_1,\ell_2, \dots, \ell_n$  linearly independent linear forms. Set  $g=\ell_1\cdots \ell_nf$ and 
denote $J(g)\subseteq R=K[x_1,\dots,x_n]$ the Jacobian ideal of $g$. Then one has:
\begin{enumerate}[\quad\rm (1)]
\item $g$ is not  a free divisor, instead
\item $\depth R/J(g)\leq \min(\max(0, n-k), n/2) < n-2$.
\end{enumerate}
In particular,   if $k\geqslant n$ then $\depth R/J(g)=0$. 
\end{theorem} 
 
Since  $k>2$ and $n>2$ implies $\max(0, n-k) < n-2$, assertion (1) follows indeed from (2)
as claimed. To prove (2) in Theorem \ref{divdiv1}, we need to set up some notation. 
To avoid confusion, $\langle a_1,\dots,a_n  \rangle$ will denote the vector with coordinates 
$a_i$, while $(a_1,\dots,a_n)$ denotes the ideal or module generated by the $a_i$. 
For a form $f$, we set $\hat f_{i}=x_{i}f_{i}+f$, with $f_{i}=\partial f/\partial  x_{i} $ as before. 

\begin{lemma} 
\label{divdiv2} 
Let $f$ be a form in $K[x_1,\dots,x_n]$. If $g=x_1\cdots x_nf$ is reduced, then the ideals 
$J(g)$ and   $(x_{i}f_{i} \ ; i=1,\dots,n)$ of $R$ have the same projective dimension. 
In particular,  $g$ is a free divisor if, and only if, $(x_{i}f_{i}\ ; i=1,\dots,n)$ is perfect of codimension $2$. 
\end{lemma}

\begin{proof}  Set  $y_i=x_1\cdots x_n/x_{i}$ and note that $g_i=y_i\hat f_{i}$. 
If $\langle \alpha_{1},\dots, \alpha_{n} \rangle$ is a syzygy  of $\nabla  g$, then 
$\langle\alpha_1 \hat f_1, \dots,  \alpha_{n} \hat f_n\rangle$ is thus a syzygy of 
$\langle y_1, \dots, y_n\rangle$. 
By the Hilbert--Burch Theorem, the syzygy module of $\langle y_1, \dots, y_n\rangle$ is generated 
by $x_1e_1-x_{i}e_i$ with $i=2,\dots,n$, whence there exist polynomials $a_2,\dots,a_n$ such that 
\begin{align*}
\alpha_1 \hat f_1&=(a_2+\dots+a_n)x_1&&\text{and}\\
\alpha_i \hat f_{i}&=-a_ix_{i} &&\text{for $i=2,\dots, n$.}
\end{align*} 
Since $g$ is  squarefree, $x_{i}$ does not divide $f$, whence that variable must divide 
$\alpha_i$  for each $i$. In other words,  $\alpha_i=x_{i}\beta_i$ for suitable $\beta_{i}\in R$, and 
then $\langle \beta_1, \dots, \beta_n\rangle$  is a syzygy of $\langle\hat f_1,  \dots,  \hat f_n\rangle$.
 
Therefore, the $R$-linear map $\psi: R^n\to R^n$ sending $e_i$ to $x_{i}e_i$ induces an isomorphism 
between the syzygy module of   $\langle\hat f_1,  \dots,  \hat f_n\rangle$ and the syzygy module of 
$\nabla g$. 

Because $f$ is homogeneous,  one has the Euler relation 
$f=\tfrac{1}{k}\sum_{i}x_{i}f_{i}$, whence 
\begin{align*}
(\hat f_{i}; i=1,...,n) \subseteq (x_{i} f_{i}; i=1,...,n)\,.
\end{align*}
Using the Euler relation once more, one obtains as well
$\sum_{i=1}^{n}\hat f_{i}= (\deg f +n)f$, thus, $f\in (\hat f_{i}; i=1,...,n)$, and then also
\begin{align*}
(x_{i} f_{i}; i=1,...,n)\subseteq (\hat f_{i}; i=1,...,n)\,. 
\end{align*}
Accordingly, these ideals agree.

It follows that the first syzygy module of the ideal $J(g)$ and that of the ideal $(x_1f_1, \dots, x_nf_n)$ 
differ only by a free summand --- whose rank is in fact the $K$-dimension of the vector space of 
Euler vector fields annihilating $f$.  So the statement follows. 
\end{proof} 

\begin{example}
\label{exdiv} 
Let us illustrate the preceding result.
\begin{enumerate}[\rm(a)]
\item Consider $f=\sum_{i=1}^k u_iM_i$ with $0\neq u_i\in K$, with $M_i$ pairwise  coprime 
monomials of same degree, and set $g=x_1\dots x_n f$. Then $\depth R/J(g)=n-k$, 
because here the ideal $(x_{i}f_{i})_{i=1,\dots, n}$ is the complete intersection ideal $(M_1,\dots,M_k)$.   
\item  Let   $f$ be the {\em Cayley form\/} in $n$ variables, the elementary symmetric polynomial of 
degree $n-1$, that can be written
\[
f = x_1\cdots x_n (x_1^{-1}+\dots+x_n^{-1})\,,
\]
and consider $g=x_1\cdots x_nf$.

Denoting $J_k$ the ideal generated by all square-free monomials of degree $k$, it is well known 
that $J_k$ is perfect of codimension $n-k+1$. The radical of the Jacobian ideal of $f$ is easily seen
to be $J_{n-2}$.  So  $f$ is irreducible and, for $n\geqslant 3$,  singular with  singular locus of
codimension $3$. 

On the other hand, one checks that $( x_{i}f_{i} ; i=1,\dots,n)=J_{n-1}$ and Lemma \ref{divdiv2} therefore
verifies that $g$ is a free divisor, as was also observed in \cite{MS}, where further a discriminant
matrix is given.

\item For a given form $f$, smooth and in generic coordinates, the elements $(x_{i}f_{i})_{i}$ tend 
to form a regular sequence.  In that case, the resolution of the first syzygy module of $J(g)$ is 
thus given by the corresponding tail of the Koszul complex on $(x_{i}f_{i})_{i}$, shifted in degree, and 
therefore $R/(x_{i}f_{i})_{i}$ embeds as the nonzero Artinian submodule 
$H^{0}_{(x_{i}; i=1,...,n)}(R/J(g))$ into $R/J(g)$, forcing
$\depth R/J(g)=0$. As a concrete example, take a Fermat hypersurface $f=\sum_{i=1}^{n}x_{i}^{k}$,
with $k\geqslant 1, n\geqslant 3$.

\item
\label{divdiv4} 
For a subset   $A$ of $\{1,\dots,n\}$, set $x_A = \Pi _{i\in A} x_{i}$. With notation as 
in Lemma \ref{divdiv2}, one obviously has 
\[
(f_{i}\ ; i\in A) \subseteq (x_{i}f_{i}\ ; i=1,\dots,n):(x_A)\,.
\]
Accordingly, either $x_A  \in (x_{i}f_{i} )_{i}$ or the projective 
dimension of  $R/(x_{i}f_{i})_{i}$ is at least the codimension of $R/(f_{i}\ ; i\in A)$.
In particular, if $\deg f>n$, then no such monomial is in 
$(x_{i}f_{i})_{i}$, and we see again that $\depth R/J(g)=0$.
\end{enumerate}
\end{example}
The last example leads to the following result.
\begin{proposition} 
\label{divdiv7}   Assume  $f\in R=K[x_1,\dots,x_n]$ with $n>2$ is smooth of degree $k>2$, and let
$\ell_1,\dots,\ell_n$ be linearly independent linear forms.  With $g=\ell_1\cdots \ell_n f$ one then has
\[
\depth R/J(g)\leqslant \max( 0,n-k)\,.
\] 
\end{proposition} 

\begin{proof}  Changing coordinates we may assume $\ell_i=x_{i}$. Set $v=\min(k,n)$. In view of 
Example \ref{exdiv}(\ref{divdiv4}) to Lemma \ref{divdiv2}, it is enough to show that 
$x_1\cdots x_v \not\in (x_1f_1,\dots,x_nf_n)$. If $k>n$ this is obvious. If $k\leq n$, then $v=k$, and 
we argue as follows. Suppose by contradiction that 
\begin{align}
\label{dis*}
\tag{$*$}
x_1\cdots x_k=\sum_i  \lambda_i x_{i} f_{i}
\end{align}
with $\lambda_i\in K$. Let $x_1^{\alpha_1}\cdots x_n^{\alpha_n}$ be a monomial in the support of 
$f$ that is different from $x_1\cdots x_k$. From (\ref{dis*}) it follows that 
$\sum_{i=1}^n \lambda_i \alpha_i=0$. If we show that the support of $f$ contains at least $n$ 
monomials different from $x_1\cdots x_k$ whose exponents are linearly independent, we can 
conclude that $\lambda_i=0$ for all $i$, thus, contradicting (\ref{dis*}).  
Since $f$ is smooth, for each $i$ there exists some $j=j(i)$, such that the  
monomial   $x_{i}^{k-1}x_{j}$ is in  the support of $f$. 

We claim that the exponents of $x_{i}^{k-1}x_{j(i)}$, for $i=1,\dots,n$, are indeed linearly independent.
To prove this, consider the linear map  $h:{\bf C}^n \to {\bf C}^n$ defined as $h(e_i)=e_{j(i)}$. Any
such map is easily seen to satisfy $(h^{n!}-1)h^n=0$, whence the eigenvalues of $h$ are either $0$ 
or roots of unity. In particular, no integer $m$ with $|m|>1$ is a root of the characteristic polynomial 
$\det(-tI+h)$ of $h$.  Therefore we have that   $\det(-tI+h)\neq 0$ at  $t=-k+1$, and this proves  the claim. 
\end{proof} 

As for a last ingredient, note the following.

\begin{lemma} 
\label{divdiv8}  If $f\in R=K[x_1,\dots,x_n]$ is smooth, then the codimension of $(x_{i}f_{i})_{i=1,...,n}$
is at least $n/2$.
\end{lemma} 

\begin{proof} Let $P$ be a minimal prime of $I=(x_{i}f_{i})_{i=1,\dots, n}$ in $R$. If $c$ is the 
number of variables $x_{i}$ contained in $P$, then that prime ideal contains at least $n-c$ of the 
$f_{i}$. Hence $P$ contains two  regular sequences: one  of length  $c$ and the other of length $n-c$. 
So the codimension of $I$ is at least $n/2$.  
\end{proof}

The 
{\em Proof of Theorem\/} \ref{divdiv1} is now obtained by combining  Lemma  \ref{divdiv2}, 
Proposition \ref{divdiv7}, and Lemma \ref{divdiv8}. \hfill\qed
\medskip

\begin{remark}
As far as  we know, in Example \ref{exdiv}(\ref{divdiv4}), it might be even true that for {\em any\/} 
smooth $f$ in {\em any\/} system of coordinates, $x_1\cdots x_n \not\in  (x_1f_1,\dots,x_nf_n)$,
so that then, in particular, always $\depth R/J(g)=0$. 

However, for a smooth $f$, the ideal $(x_{i}f_{i})_{i=1,\dots, n}$ can be of codimension 
$n/2$, but, of course, only for $n$ even. For example,
\[
f=(x_1^{k-1}+x_2^{k-1})x_2+(x_3^{k-1}+x_4^{k-1})x_4
\]
is smooth and the codimension of $(x_{i}f_{i})_{i=1,\dots, 4}$ is $2$. Nevertheless, 
in this case $R/(x_{i}f_{i})_{i=1,\dots, n}$ still has depth $0$ since 
$x_1x_2x_3x_4\not\in (x_{i}f_{i})_{i=1,\dots, n}$. 
\end{remark}

\section{Extending Free Divisors into the Tangent Bundle} 
Let $R=K[x_1,\dots,x_n]$ as before, and set $R'=R[y_1,\dots,y_n]$. Define a map $^*:R\to R'$ by   
$$f^*=\sum_{i=1}^n y_i \partial f/\partial x_i$$
for every $f\in R$.   Clearly $^*$ is a $K$-linear derivation. For a matrix $C=(c_{ij})$ with entries in 
$R$ we set $C^*=(c_{ij}^*)$. 

\begin{theorem}
\label{jets}
Let $f\in R$ be a homogeneous free divisor of degree $k>0$. Then $ff^*$ is a free divisor in $R'$, in 
$2n$ variables 
and of total degree $2k$, that  is linear if $f$ is so. 
\end{theorem}

\begin{proof} First note  that $ff^*$ is reduced because  $f^*$ is irreducible. By contradiction, if $f^*$ 
were reducible then, since $f^*$ is homogeneous of degree $1$ in the $y$'s,  the partial derivatives 
of $f$ had a non-trivial common factor contradicting the fact that $f$ is reduced.  

Secondly we identify a discriminant matrix for $ff^*$. Since $f$ is homogeneous, a discriminant 
matrix for $f$ can be constructed as follows. Because $J(f)$ is a perfect ideal of codimension $2$, 
we can find a Hilbert-Burch matrix $B=(b_{ij})$ for $J(f)$, of size  $n\times (n-1)$, such that the 
$(n-1)$-minor of $B$ obtained by removing the $i$-th row is 
$(-1)^{i+1}\partial f/\partial x_i$. 

Adjoining $x^T=(x_1,\dots,x_n)^T$ as a column to the matrix $B$, we obtain the matrix
$$
A =(B\mid x^T)
$$
that is by construction a discriminant matrix for $f$. We now claim that  the following $2n\times 2n$ 
block matrix 
$$
A'= \left(\begin{array}{ccccc}
B  & x^T &\vline & 0 & 0 \\
\hline
\vphantom{y^{T^{1}}}B^*& 0   &\vline  & B & y^T
\end{array}\right) 
$$
is a discriminant matrix for $ff^*$.
Its determinant is clearly $ff^*$ by definition of $A, B$ and $f^{*}$. 
The product rule yields
\begin{align*}
\nabla(ff^*)&=f^*(\nabla_x(f),0)+f(\nabla_x(f^*), \nabla_x(f))\,,
\intertext{and hence} 
\nabla(ff^*)A'&=f^*(\nabla_x(f),0)A'+f(\nabla_x(f^*),\nabla_x(f))A'\,.
\end{align*}
Now $(\nabla_x(f),0)A'=(\nabla_x(f)A,0)\equiv 0 \bmod (f)$, and so it remains to show that 
\begin{align}
\label{disff*}
\tag{$\dagger$}
(\nabla_x(f^*),\nabla_x(f))A'\equiv 0 \bmod (f^*)\,.
\end{align}
Expanding returns the vector
\begin{align*}
(\nabla_x(f^*),\nabla_x(f))A'=(\nabla_x(f^*)B+\nabla_x(f)B^*, \nabla_x(f^*)x^T, \nabla_x(f)B,  
\nabla_x(f)y^T)\,.
\end{align*}
Concerning its first part, note that $\nabla_x(f^*)=\nabla_x(f)^*$, whence
\begin{align*}
\nabla_x(f^*)B+\nabla_x(f)B^* &= (\nabla_x(f)B)^*&&\text{because $^*$ is a derivation,}\\
&=0^{*}=0&& \text{as $\nabla_x(f)B=0$ by construction.} 
\intertext{Regarding the second component,}
\nabla_x(f^*)x^T &=(k-1)f^{*}\equiv 0 \bmod(f^*)\,,
\intertext{because $f^*$ is homogeneous of degree $k-1$ with respect to the variables $x$. 
Finally,}
\nabla_x(f)B&=0&&
\text{by choice of $B$, and}\\  
\nabla_x(f)y^T&=f^*&&\text{by definition.}
\end{align*}
Therefore, (\ref{disff*}) holds and $ff^{*}$ is confirmed as a free divisor. The assertions on degree
and number of variables are obvious from the construction.

A free divisor is linear if all entries in a discriminant matrix are linear, and this property 
is clearly inherited by $A'$ from $A$.
\end{proof} 
 
\begin{remark}
The geometric interpretation of the hypersurface defined by $ff^{*}$ is as follows.

Viewing $f\in R$ as the function $f\colon \Spec R=\AA_{K}^{n}\to \AA_{K}^{1}=\Spec K[t]$, 
its differential fits into the exact Zariski--Jacobi sequence of K\"ahler differential forms
\[
\xymatrix{
0&\ar[l] \Omega^{1}_{K[t]/R}&\ar[l]\Omega^{1}_{R/K}\cong \oplus_{i}Rdx_{i}&&
\ar[ll]_-{df\partial/\partial t}\Omega^{1}_{K[t]/K}\otimes_{K[t]}R\cong Rdt}
\]
and one may interpret $R'\cong \Sym_{R}\Omega^{1}_{R}$ as the ring of 
regular functions on the tangent bundle $T_{X} \cong \Spec R'\cong \AA^{2n}_{K}$ over 
$X=\Spec R\cong \AA^{n}_{K}$. 

This identifies $R'/(f^{*})$ with the regular functions on the total space of the affine relative tangent 
``subbundle'' $T_{X/S}\subseteq T_{X}$, the kernel of the Jacobian map $df:T_{X}\to T_{S}$ that 
consists of the vector fields vertical with respect to (the fibres of) $f$ over the affine line $S=\Spec K[t]$. 

Accordingly, the hypersurface $H$ defined by $ff^{*}$ is the union of that affine ``bundle'' with 
$\Spec R'/(f)$, the restriction of the total tangent bundle $T_{X}$ to 
$\Spec R/(f)$, in turn the fibre over $0$ of the function $f$. Equivalently, $\Spec R'/(f)$ is the 
suspended free divisor obtained as the inverse image of $\Spec R/(f)$ along the structure morphism
$p: T_{X}\to X$. Thus, $H=T_{X/S}\cup \Spec R'/(f)= df^{-1}(0)\cup(fp)^{-1}(0)\subseteq T_{X}$.

\begin{align*}
\xymatrix{
&\ar@{^{(}->}[dl]T_{X/S}\ar@{_{(}->}[rd]\ar[rr]&&\{0\}\ar@{_{(}->}[rd]\\
\ \hphantom{H}H\ar@{^{(}->}[rr]&&T_{X}\ar[rr]^-{df}\ar[rd]&&T_{S}\ar[rd]\\
&\ar@{_{(}->}[ul](fp)^{-1}(0)\ar[rd] \ar@{^{(}->}[ru]&&X\ar[rr]^-{f}&&S\\
&&f^{-1}(0)\ar@{^{(}->}[ru]\ar[rr]&&\{0\}\ar@{^{(}->}[ru]
}
\end{align*}
Interesting examples are hard to visualize as they will live in four or more dimensions. However, the
intersection of the two (unions of) components,  $T_{X/S}\cap \Spec R'/(f)\subseteq \Sing H$ is easy to 
understand: Geometrically, over $X$ it fibres into the union of the hyperplanes perpendicular to 
$\nabla f(x)$ for some $x\in X$ on $\{f=0\}$, that is,
\begin{align*}
T_{X/S}\cap \Spec R'/(f) = 
\bigcup_{x, f(x)=0}\left\{(x,y)\in \AA^{n}\times\AA^{n}\mid \nabla f(x)y=0\right\}\,.
\end{align*}
\end{remark}

\begin{example}   
Applying Theorem \ref{jets} to the normal crossing divisor  $x_1\cdots x_n$ we find that
$$(x_1\cdots x_n)^2\sum_{i=1}^n  \frac{y_i}{x_i}$$ 
is a linear free divisor.   
\end{example} 
 
\begin{remarks}  
\label{remff*}
Various generalizations are possible:
\begin{enumerate}[\rm (1)] 
\item
\label{remff*1}
Given a homogeneous free divisor $f$ in a polynomial ring of dimension $n$, 
one can iterate the use of Theorem \ref{jets} to get an infinite family 
$\{ F_i\}_{i\in \NN}$ of homogeneous free divisors, defined by $F_0=f$ and $F_{i+1}=F_iF_i^*$,
where $^{*}$ is, of course, to be understood relative to the polynomial ring containing $F_{i}$. 
 By construction, $F_i$ belongs to a polynomial ring of dimension $2^in$, its degree equals 
$(i+1)\deg f$, and it is a linear free divisor if, and only if, $f$ is linear. 

Taking $F_{0}=x$ as a seed, we obtain the sequence of linear free divisors
\begin{gather*}
x\,, xy\,, xy(xz_{1}+yz_{2})\,, \\
xy(xz_{1}+yz_{2})(2xyz_{1}u_{1}+y^{2}z_{2}u_{1}+ x^{2}z_{1}u_{2}+2xyz_{2}u_{2}
+x^{2}yu_{3}+xy^{2}u_{4})\,,...
\end{gather*}
in $K[x,y,z_{1}, z_{2}, u_{1},..., u_{4},...]$.
\item
\label{remff*2}
Theorem \ref{jets} holds also for free divisors that are weighted homogeneous of degree $d\neq 0$
with respect to  some weight vector $w=(w_1,\dots,w_n)\in \ZZ^n$. In the proof one simply replaces 
the column vector $x^T$ in the discriminant matrix with $(w_1x_1,\dots,w_nx_n)^T$. 
Again, linearity is preserved.
\end{enumerate}
 \end{remarks} 

One can further generalize Theorem \ref{jets}, as well as Remark \ref{remff*}(\ref{remff*1}), also 
as follows, incorporating right away the weighted homogeneous version as in 
Remark \ref{remff*}(\ref{remff*2}).

\begin{theorem}
With notation as before, assume $f$ weighted homogeneous of degree $d\neq 0$
with respect to  some weight vector $w=(w_1,\dots,w_n)\in \ZZ^n$.

With $m\geqslant 1$, let $R'=R[y_{ij} : 1\leq i\leq n, 1\leq j\leq m]$, assign weights $|y_{ij}|=w_{i}$, 
and set $f^{\{*_{j}\} }=\sum_i y_{ij}\partial f/\partial x_i$. Then $f\prod_{j=1}^m f^{\{*_{j}\}}$ is a 
free divisor in $(m+1)n$ variables of weighted homogeneous degree $(m+1)d$ that will be 
linear along with $f$.
\end{theorem}

\begin{proof}
The proof is a simple variation of the one given for $m=1$. For instance, if $m=2$, the discriminant
matrix can be taken as
$$
\left(\begin{array}{ccccccccc}
\vphantom{B^{T^{\{1\}}}}B  & wx^T &\vline & 0  & 0    &\vline &      0 & 0 \\
\hline
\vphantom{B^{T^{\{1\}}}}B^{\{*_{1}\}}& 0   &\vline  & B & wy_1^T &\vline & 0 & 0 \\
\hline
\vphantom{B^{T^{\{1\}}}}B^{\{*_{2}\}}& 0   &\vline  &  0 & 0&\vline  & B & wy_2^T
\end{array}\right) 
$$
where $wx=(w_1x_1,\dots,w_nx_n)$, with $wy_{1}, wy_{2}$ analogous abbreviations.
\end{proof}

In this way, one may obtain any normal crossing divisor $x_{0}\cdots x_{m}$, starting from 
$f=x_{0}$ and using $f^{\{*_{j}\} }= x_{j}\partial f/\partial x_{0}=x_{j}$ for $j=1,...,m$.

\end{document}